\documentclass[12pt, reqno]{amsart}
\usepackage{amsfonts,amsmath,amssymb,amsthm,epsfig, cite, graphicx, hyperref,
	color, esint,fancyhdr, enumerate, latexsym,amsrefs, makecell}
\usepackage{color}
\usepackage[mathscr]{eucal}
\usepackage{comment}
\usepackage{xcolor}

\usepackage[nameinlink]{cleveref}

\numberwithin{equation}{section}

\theoremstyle{definition}
\newtheorem{definition}{Definition}[section]

\theoremstyle{remark}
\newtheorem{remark}[definition]{Remark}
\theoremstyle{plain}
\newtheorem{theorem}[definition]{Theorem}
\newtheorem{lemma}[definition]{Lemma}

\newtheorem{claim}[definition]{Claim}
\newtheorem{result}[definition]{Result}

\newcommand{\eps}{\varepsilon}

\newcommand{\Om}{\Omega}

\newcommand{\st}{\subset}
\newcommand{\St}{\Subset}

\newcommand{\mco}{\mathcal{O}}

\newcommand{\bn}{\mathbb{B}^{d}}

\newcommand{\cplx}{\mathbb{C}}

\newcommand{\re}{\mathbb{R}}
\newcommand{\cn}{\mathbb{C}^d}

\begin{document}
	\title[On the Approaching Geodesic Property via the Quotient Invariant]{On the Approaching Geodesic Property via the Quotient Invariant}
	\author{ Kingshook  Biswas AND Sanjoy Chatterjee}
    \address{Stat-Math Unit, Indian Statistical Institute, 203 B. T. Rd., Kolkata 700108, India}
    \email{kingshook@isical.ac.in}
	\address{Department of Mathematics and Statistics, Indian Institute of Technology,
				KANPUR}
		\email{ramvivsar@gmail.com }

	\date{\today}
\begin{abstract}
 We study the approaching geodesic property of  a bounded domain in $\cn$ with respect to the Kobayashi distance  using the quotient invariant.
\end{abstract}

\maketitle

\section{Introduction}

The Gromov compactification $\overline{X}^{G}$ of a proper geodesic Gromov hyperbolic metric space $(X,d)$ is fundamental in the study of metric geometry. In \cite{Gromov:1981}, Gromov introduced a compactification procedure for proper metric spaces by embedding the metric space $(X,d)$ into the space $C(X)$ of real-valued continuous functions (see Section~\ref{sec-prem} for the precise definition). This construction, commonly referred to as the horofunction compactification and denoted by $\overline{X}^{H}$, provides a natural boundary associated to the metric structure of $X$. A fundamental question is to understand how these two compactifications  $\overline{X}^{H}$ and $\overline{X}^{G}$ compare to one another. In \cite[Theorem 1.1]{AFG2024}, it is shown that for a proper geodesic Gromov hyperbolic metric space $(X,d)$, the \emph{approaching geodesic property} (see Definition~\ref{def-apgeo}) suffices to guarantee that the horofunction compactification $\overline{X}^{H}$ and the Gromov compactification $\overline{X}^{G}$ are homeomorphic. More precisely, the following result is proved in \cite[Theorem 1.1]{AFG2024}:

\begin{result}\label{R:extn of busemanmap}
Let $(X , d)$ be a proper geodesic Gromov hyperbolic metric space with
the approaching geodesics property. Then the Busemann map induces a homeomorphism $\overline{X}^{H} \to \overline{X}^{G}$ extending the identity map. Thus   $\overline{X}^{H}$ and $\overline{X}^{G}$ are topologically equivalent.

\end{result}

The approaching geodesic property also has further implications in the study of orbits under the iteration of non-expanding self-maps of Gromov hyperbolic metric spaces; see, for instance, \cite[Proposition 4.20]{AroFia2025} and \cite[Proposition 4.10]{AroFia2025}. It follows from \cite[Proposition 5.2]{AroFia2025} that not every Gromov hyperbolic metric space satisfies the \emph{approaching geodesic property}. This naturally raises the question: which metric spaces do possess the \emph{approaching geodesic property}? The study of the metric geometry of a domain in $\mathbb{C}^n$ endowed with the Kobayashi distance is one of the major areas of interest in several complex variables. In this article, we always consider the domains in $\cn$ for which the Kobayashi (pseudo)distance is a genuine distance; that is, we consider Kobayashi hyperbolic domains in $\cn$.
 In \cite{AFG2024}, Arosio et al. identified a class of Kobayashi hyperbolic domains in $\cn$ that satisfy the approaching geodesic property  using  the squeezing function. They established the following result:

\begin{result}\label{R:appsqu}\cite[Theorem 4.3]{AFG2024}
    Let $D \subset \mathbb{C}^n$ be a bounded domain such that $\lim_{z \to \partial D} s_{D}(z) = 1$. Then the Kobayashi metric space $(D, K_{D})$ has the approaching geodesic property.
\end{result}
\noindent
As a consequence of Result~\ref{R:appsqu}, Arosio et al.\ proved the following:

\begin{result}\cite[Corrolary 4.5]{AFG2024}\label{R:existennce of limit}
Let $D \Subset \mathbb{C}^n$ be one of the following:
\begin{enumerate}
    \item[(i)] a strongly pseudoconvex domain with $\mathcal{C}^{2}$ boundary,
    \item[(ii)] a convex domain satisfying $\lim_{z \to \partial D} s_{D}(z) = 1$.
    \end{enumerate}
Let $K_{D}$ denote the Kobayashi distance on $D$. Then any two of the compactifications $\overline{D}^{H}$, $\overline{D}$, and , $\overline{D}^{G}$ are homeomorphic via maps extending the identity. 
\end{result}
The existence of a  continuous extension of the identity map from the Euclidean closure $\overline{D}$ to the horofunction compactification  $ \overline{D}^{H}$ ensures  the existence of the  following limit for any $\zeta \in \partial D$ :  
\begin{align}\label{E:limit}
   h_{p,\zeta}(z):= \lim_{w\to \zeta} [K_{D}(z,w)-K_{D}(p, w)].
\end{align}
\noindent
If the limit in \eqref{E:limit} exists then the corresponding horoballs are defined by
the set 
$\{ z\in D : h_{p,\zeta}(z) < \tfrac{1}{2}\ln R \}.
$
The notion of a horoball plays an important role in the study of the dynamics of iterates of holomorphic self-maps (see \cite{Abate1988Z}, \cite{Abateiterationtheorynote}, \cite{Abate1990} for details ). Therefore, the study of approaching geodesic property of Kobayashi hyperbolic domain is important. The aim of this article is  to study the   approaching geodesic property of the domain in terms of the  {\em quotient invariant } which is a biholomorphic invariant real valued function defined on a domain $\cn$. Let us mentioned the definition of quotient invariant before going to our main result.
\begin{definition}
Let $D \subset \cn$ be a  domain. The Carathéodory and
Kobayashi–Eisenman volume elements on $D$ at the point $p \in D$ are denoted by $c_{D}(p)$ and $k_{D}(p)$ respectively, and is defined by the following:
\begin{align*}
    c_{D}(p)&:=\sup\{|\text{det}\psi'(p)|^2: \psi \in \mathcal{O}(D, \mathbb{B}^{n}), \psi(p)=0\},\\
     k_{D}(p)&:=\inf\{|\text{det}\phi'(0)|^{-2}: \phi \in \mathcal{O}(\mathbb{B}^{n}, D), \phi(0)=p\}.
\end{align*}
\end{definition}
 \begin{remark}\label{remark}
    For any  taut domain $D \st \cn $ ,  by Montel's theorem, it follows there exists $\phi \in \mathcal{O}(\mathbb{B}^{n}, D)$ such that $k_{D}(p)=|\text{det}(\phi'(0))|^{-2} \neq 0$. Similarly,  there exists  $\psi \in \mathcal{O}(D,\mathbb{B}^{n})$ such that $c_{D}(p)=|\text{det}(\psi'(p))|^{2}$ for some $\psi \in \mathcal{O}(D,\mathbb{B}^{n})$.
    \end{remark}
\noindent
The {\em quotient invariant} of taut domain $D$ is denoted by $q_{D}(z)$ and is  defined by  
\begin{align}
    q_{D}(z):=\frac{c_{D}(z)}{k_{D}(z)}.
\end{align}
It is well known that $q_{D}(z)$ is a biholomorphic invariant \cite[Chapter 11, p.~446]{Krantzbook}. 
It follows from \cite[Theorem 1]{GWU1985} that if there exists a point $p \in D$ such that 
$q_{D}(p)=1$, then $D$ is biholomorphic to the unit ball. 
In \cite[Theorem~1.3]{DevBorah2020}, the author has used the quotient invariant to detect 
strong pseudoconvexity at boundary points of a domain. In this paper, we have studied the approaching geodesic property of a domain via quotient invariant.

The main result of the paper is the following:

\begin{theorem}\label{T:approachsgeo}
Let $D \St \cn$ be a complete hyperbolic domain and  $\zeta \in \partial_{H}D$. Suppose that $q_{D}(z) \to 1$ as $z \to \zeta$. Then the domain $D$ has the approaching geodesic property at $\zeta$.
\end{theorem}
\begin{remark}\label{remark1}
By the Schwarz lemma, one has $q_{D}(z) \le 1$ for all $z \in D$. It follows from  \cite[Theorem 3.1]{DGZ2016} that  for any bounded domain  $D \subset \cn$   $s_{D} \leq q_{D}(z)$. Hence, $s_{D}(z) \to 1$ does not necessarily imply that that $q_{D}(z)\to 1 $.
\end{remark}
\noindent
It follows from Remark~\ref{remark1} that Theorem~\ref{T:approachsgeo} extends the Result~\ref{R:appsqu}. Next, using Theorem~\ref{T:approachsgeo}, we establish a metric-geometric property of bounded convex domains for which $q_{D}(z) \to 1$ as $z \to \partial D$. In particular, we show that 
\begin{theorem}\label{T:horosphere}
If  $D\Subset \cn$ be a convex domain and  $\lim_{z \to  \partial D}q_{D}(z)=1$. Then $(D, K_{D})$ is a Gromov hyperbolic metric space that enjoys the visibility property. Moreover, any two compactification $\overline{D}$, $\overline{D}^{G}$ and $\overline{D}^{H}$ are homeomorphic via the  maps extending the identity.
\end{theorem}

\section{Preliminaries and Technical Results}\label{sec-prem}
In this section, we introduce the definitions and notions relevant to this paper, and recall several known results that will be used in the proofs of our  theorems.
\begin{definition}[Gromov compactification]
Let $(X,d)$ be a proper geodesic Gromov hyperbolic metric space and  $\mathcal{R}(X)$ denotes the set of all geodesic rays in $(X, d)$. Define an equivalence relation on 
$\mathcal{R}(X)$ as follow:   $\gamma \sim \sigma $ if and only if $\sup_{t \geq 0}d(\gamma(t), \sigma(t))<\infty$. Two geodesic  rays $ \gamma , \sigma$ are said to be   {\em asymptotic} if $\gamma \sim \sigma$. The {\em Gromov boundary} of $(X, d)$ is denoted by $\partial_{G}X$ and is defined by $\partial_{G}X:=\mathcal{R}(X)/\sim$. The {\em Gromov compactification} of the metric space $(X, d)$ is defined by $\overline{X}^{G}:=X\sqcup \partial_{G}X$ endowed with a compact metrizable topology defined as \cite{AFG2024} 
    \end{definition}

Next, we mention the notion of horofunction compactification introduced by Gromov in \cite{Gromov:1981}

\begin{definition}[Horofunction compactifiation]
Let $(X,d)$ be a proper metric space. Let $C(X)$ be the set of all real-valued  continuous functions on $X$ endowed with the topology of uniform convergence. Fix a base point $p \in X$. For each $x \in X$ the map $d_{x}:X \to 
 \mathbb{R}$ is defined  by $d_{x}(z)=d(z,x)-d(p,x)$. Clearly, them  $i_{H}: X \to C(X)$ defined by   $x \mapsto d_{x}$ induces a homeomorphic  embedding of $X$ in the
subspace of $C(X)$ consisting of functions vanishing at the point $p \in X$. The {\em horofunction compactification } of $X$ is  denoted by $\overline{X}^{H}$ and is defined by closure of the image of  $X$ in $C(X)$ under the map  $i_{H}$. The {\em horofunction
boundary} of $X$ is defined by $\partial _{H} X:=\overline{X}^{H} \setminus X$. For a geodesic ray $\gamma:[0, \infty) \to X$ we write $\gamma(\infty)=\zeta \in \partial_{H}X$ if the $\lim_{t\to \infty}{[d(z,\gamma(t))-t]}$ exists for all $z \in x$ and  $\zeta(z)=\lim_{t\to \infty}{[d(z,\gamma(t))-t]}$.

    
\end{definition}

\begin{definition}\label{def-apgeo}

Let $(X, d)$ be a proper geodesic metric space. Two geodesic rays $\gamma, \sigma \in \mathcal{R}(X)$ are {\em strongly asymptotic} if there exists $T \in \mathbb{R}$ such that 
$$ \lim_{t \to \infty}d(\gamma(t), \sigma(t+T))=0 .$$

A proper geodesic metric space is $(X,d)$ said to have {\em the approaching geodesic property} if every {\em asymptotic geodesic} ray are {\em strongly asymptotic }. If $(X, d)$ is  proper    metric space and $\zeta \in \partial_{H}X$ we say that $X$ has {\em approaching geodesic at $\zeta$} if any two asymptotic geodesic rays $\gamma, \sigma $ with $\gamma(\infty)=\sigma(\infty)=\zeta$ are strongly asymptotic.
\end{definition}

\begin{definition}
Let   $D \st \cn$, $ z \in D$ and $X \in \mathbb{C}^n$. The Kobayashi-Royden (pseudo)metric $\kappa_D$  of $D$ are defined as:
\[
\kappa_D(z; X) = \inf \{ |\alpha| : \exists \varphi \in \mathcal{O}(D, D), \, \varphi(0) = z, \, \alpha \varphi'(0) = X \}.
\]
For $x, y \in D$ the Kobayashi distance,  is defined as follows :
\begin{equation}
    K_D(x, y) = \inf_\gamma \int_0^1 \kappa_D(\gamma(t); \gamma'(t)) \, dt,
\end{equation}
where the infimum is taken over all piecewise $C^1$ curves $\gamma : [0, 1] \to D$ with $\gamma(0) = x$ and $\gamma(1) = y$.
\end{definition}
A \textit{geodesic} for $K_D$ is a curve $\sigma : I \to D$, where $I$ is an interval in $\mathbb{R}$, such that for any $s, t \in I$, 
\[
K_D(\sigma(s), \sigma(t)) = |s - t|.
\]
Here we mention some known terminology:
\begin{itemize}
    \item If $x, y \in D$, $I = [0, a]$ and $\sigma(0) = x$, $\sigma(a) = y$, then $L = K_D(x, y)$ and we say that $\sigma$ is a geodesic joining $x$ and $y$.
    \item If $I = (-\infty, +\infty)$, we say that $\sigma$ is a \textit{geodesic line}.
    \item If $I = [0, \infty)$, we say that $\sigma$ is a \textit{geodesic ray}.
\end{itemize}

A geodesic ray $\sigma$ \textit{lands} if there exists $p \in D$ such that $\lim_{t \to \infty} \sigma(t) = p$.

The domain $D$ is \textit{complete hyperbolic} if it is Kobayashi hyperbolic and if $K_D$ is a complete distance, or equivalently, the balls for the Kobayashi distance are relatively compact. Bounded convex domains are well known to be complete hyperbolic. If a domain $D$ is complete hyperbolic, by the Hopf-Rinow Theorem, $(D, K_D)$ is a geodesic space, namely, any two points in $D$ can be joined by a geodesic.

We now mention some known results that will be used to prove our main theorems. Before going to state the result, let us mention some required definitions. 
A domain $D \st \cn$ is said to have a simple boundary if every holomorphic map $\phi: \mathbb{D} \to \partial\Om$ is constant. 

A convex domain $D \st \cn$ is called $\cplx$-properly convex if $D$ does not contain any entire complex affine lines. The set of all $\cplx$ properly convex domains  endowed with  local Hausdorﬀ topology (see \cite[Section 3]{Zimmersubelip2022} for details) is denoted by $\mathbb{X}_{d}$. The set $\mathbb{X}_{d,0}:=\{(D,z)\in \mathbb{X}_{d} \times \cn: z \in D\}$  is endowed with the subspace topology. With this terminology in place, we now state the following results, which will be used in the proof of Theorem~\ref{T:horosphere}.

\begin{result}\cite[Theorem 1.5]{Zimmersubelip2022}\label{R:Zhimmersimpleboundary}
 Let  $D \st \cn$  be a bounded convex domain. Then $(D, K_{D})$
is Gromov hyperbolic if and only if every domain in
$\overline{\text{Aff}(\cn)\cdot D} \cap \mathbb{X}_{d}$ 
has a simple boundary. Here,  $\text{Aff}(\cn)$ denotes the set of all affine automorphisms of $\cn$. 
\end{result}

\begin{result}\cite[Theorem 1.3]{DevBorah2020}\label{R:upper semicon}
   The function $f\colon \mathbb{X}_{d,0} \to \mathbb{R} $ define by $f(D,z)=q_{D}(z)$ is upper semicontinuous.   
\end{result}

\begin{result}\cite[Theorem 3.3]{BNT2022}\label{R:bracci2}
    Let $D$ be a complete hyperbolic bounded domain. Assume that $(D, K_{D})$ is 
Gromov hyperbolic. Then $D$ has the visibility property if and only if the 
identity map extends as a continuous surjective map 
$\Phi : \overline{D}^{G} \longrightarrow \overline{D}.$
Moreover, $\Phi$ is a homeomorphism if and only if $D$ has no geodesic loops in $D$. 
\end{result}

\begin{result}\cite[Theorem 1.5]{BGZ2021}\label{R:bracci1}
    Let $D \St \cn$ be a  convex domain . If $(D, K_{D})$ is Gromov hyperbolic, then the identity map 
$
id : D \to D
$
extends to a homeomorphism
$
\overline{id} : \overline{D}^{G} \to \overline{D}.
$
\end{result}

\noindent
The following result gives a necessary and sufficient condition for two geodesic rays to be strongly asymptotic.

\begin{result}\label{R:appgeocondition}\cite[Proposition 3.13]{AFG2024} Let $(X, d)$ be a proper metric space and $\gamma$ and $\tilde{\gamma}$ be two geodesic rays. The rays $\gamma$ and $\tilde{\gamma}$ are strongly asymptotic if and only if
$\lim_{t \to +\infty} \inf_{s\geq 0} d(\gamma(t), \tilde{\gamma}(s)) = 0$.
\end{result}

\noindent
The next result will be used crucially in the proof of Theorem~\ref{T:approachsgeo}.
\begin{result}\cite[Lemma 15.3.2]{Rudin2008}\label{R:Rudin}
If $F:\bn \to \bn$ is holomorphic, $F(0)=0$, $\text{det}(DF(0))=1$, then $F$ is unitary.  In particular $F \in \text{Aut}(\bn)$. 
\end{result}

\section{Proof of Theorem~\ref{T:approachsgeo}}

In order to prove Theorem~\ref{T:approachsgeo}, we first establish several preparatory lemmas. We use the following notation throughout the proof
 $B_{R}:=\{z \in \bn: K_{\bn}(0,z)<R\}$.
For $S \subset D$ we write $B_{D}(S, \eps):=\{z \in D: \inf_{x \in S}K_{D}(z,x)<\eps\}$.
\begin{lemma} \label{L:containmentball}
    Let $D \St \cn$ be complete hyperbolic domain and $\phi \in \mco(\bn, \cn)$. Let $R, \eps, \delta $ be positive real numbers such that $R>\eps>\delta>0 $ and $\phi$ be an open map on $B_{R}$. Let $K_{D}(\phi(z), \phi(0))>K_{\bn}(0,z)-\delta$ for all $z \in \overline{B}_{R}$  Then $B_{D}(\phi(0), R-\eps) \st B_{D}(\phi(B_{R}), \eps)$. 
\end{lemma}
\begin{proof}
    Clearly, $B_{D}(\phi(0), R-\eps)$ and $B_{D}(\phi(B_{R}), \eps)$ both are domains. Let $z \in B_{D}(\phi(0), R-\eps)$ and $z \notin B_{D}(\phi(B_{R}), \eps)$. Let $\gamma: [0, 1] \to D$ be a geodesic joining $\phi(0)$ and $z$. Since $\gamma(0)$ is inside the domain $B_{D}(\phi(B_{R}), \eps)$ and $\gamma(1)$ is outside the domain hence there exists $t_{0} \in (0,1)$ such that $\gamma(t_{0}) \in \partial B(\phi(B_{R}), \eps)$. Now we have the following: 
    \begin{align}\label{E:con1}
        K_{D}(\phi(0), z) >t_{0}\geq \text{dist}(\phi(0), \partial B(\phi(B_{R}), \eps)) \geq \text{dist}(\phi(0), \partial \phi(B_{R})).
    \end{align}
\noindent
Here $\partial \phi(B_{R})$ is compact.  Hence, there exists $w \in \partial \phi(B_{R})$ such that $\text{dist}(\phi(0), \partial \phi(B_{R}))=K_{D}(\phi(0), w)$. Since $w \in \partial \phi(B_{R})$, there exists $z_{n} \in B_{R}$ such that $\phi(z_{n}) \to w$ as $n \to \infty$. Since $B_{R}$ is compact subset of $\bn$, we can assume  that $z_{n} \to z_{0}$ in $\overline{B}_{R}$. Since $\phi$ is an open map, it follows that $z_{0} \in \partial B_{R}$. Therefore, $w=\lim_{n \to \infty} \phi(z_{n})=\phi(z_{0})$ for some $z_{0} \in \bn$ with $K_{\bn}(0, z_{0})=R$. Hence, from \eqref{E:con1} and our assumption we obtain that 
\begin{align}
   R-\eps> K_{D}(\phi(0),z)&> \text{dist}(\phi(0), \partial \phi(B_{R})) =K_{\bn}(\phi(0),\phi(z_{0})) \notag \\
   \intertext{From our assumption we have that }
 R-\eps> K_{D}(\phi(0),z)   &\geq K_{\bn}(0,z_{0})-\delta=R-\delta .\notag
     \end{align}
\noindent
This leads to a contradiction with our assumption of $\eps>\delta$.
\end{proof}

\begin{lemma}\label{L:containofgeo}
    Let $a \in \mathbb{R}$ and $\gamma :[a, \infty) \to D$ be a geodesic ray. Let $\phi \in \mathcal{O}(\bn , \cn)$ be an open map such that for any $R> \eps> \delta>0$, the assumption of Lemma~\ref{L:containmentball} holds. If $\gamma(t_{0})=\phi(0)$ for some $t_{0} \in [a, \infty)$. Then for any compact subset $K \subset [a, \infty)$ there exists $R_{K}>0$ such that $\gamma(K) \st B_{D}(\phi(B_{R_{K}}), \eps)$.
\end{lemma}

\begin{proof}
   Let $K \st [a, \infty)$ be any compact subset and $r:=\sup_{t\in K}|t|>0$. Choose $R_{K}:=r+\eps+|t_{0}|$. We obtain the following for all $t \in K$
   \begin{align}
       K_{D}(\gamma(t), \phi(0))=K_{D}(\gamma(t),\gamma(t_{0}))=|t-t_{0}|\leq |t|+|t_{0}|\leq R_{K}-\eps.
   \end{align}
\noindent
Therefore, in view of Lemma~\ref{L:containmentball}, we obtain that $\gamma(t) \in B_{D}(\phi(B_{R_{K}}), \eps)$ for all $t \in K$.
\end{proof}

\begin{lemma}\label{L:convergeo}
    Let $t_{n} >0$ with $t_{n} \to \infty$ as $n \to \infty$ and  $\gamma_{n}:[-t_{n}, \infty) \to D$ be a sequence of geodesic rays such that $\gamma_{n}(0)=z_{n}$ and $z_{n} \to \partial D$ as $n \to \infty$. Let  $\psi_{n} \in \mco(D, \bn)$ with 
 $\psi_{n}(z_{n})=0$ and $\phi_{n} \in \mco(\bn, D)$ with $\phi_{n}(0)=z_{n}$. If  $\psi_{n}\circ
  \phi_{n} \colon \bn \to \bn$ converges to $F \in \text{Aut}(\bn)$ uniformly over every compact subsets of $\bn$. Then $\psi_{n} \circ \gamma_{n} \to \eta$ (upto a subsequence ) uniformly over every compact subsets of $\mathbb{R}$ to a geodesic line $\eta: \mathbb{R} \to \bn$.
\end{lemma}

\begin{proof}
Let $\eps>0$,  $K \st \mathbb{R}$ be any compact subset and $r=\sup_{t \in K} |t|>0$. Clearly, there exists $n_{K} \in \mathbb{N}$ such that $K \st \text{domain}(\gamma_{n})$ for every $n \geq n_{K}$. By the Arzelà-Ascoli theorem, the sequence $\psi_{n} \circ \gamma_{n}$ admits a subsequence that converges uniformly on $K$. Renaming the subsequence, we assume without loss of generality that $\psi_{n} \circ \gamma_{n}$ converges uniformly on $K$ as $n \to \infty$. Since the compact set is arbitrary, we  define  $\eta: \mathbb{R} \to \bn$  by $\eta:=\lim_{n \to \infty} \psi_{n}\circ \gamma_{n}$, where the convergence is uniform over every compact subset of $\re$. We show that $\eta $ is a geodesic ray. To this aim, we prove the following claim:

\begin{claim}\label{Claim1}
     For any $R> \eps> \delta>0$, there exists $n_{R,\eps} \in \mathbb{N}$ such that
     \begin{align}\label{E:claim1}
         K_{\bn}(x,y) \geq K_{D}(\phi_{n}(x),\phi_{n}(y)) \geq K_{\bn}(x,y) -\delta~~\forall x,y \in \overline{B_{R}},~~\forall n \geq n_{R, \eps}.
     \end{align}
     \begin{proof}
      Here  $\overline{B}_{R}$ is a compact subset of $\bn$. From the assumption, we conclude that $\psi_{n}\circ \phi_{n} \to F \in \text{Aut}(\bn)$ uniformly over $\overline{B}_{R}$. From  the continuity of Kobayashi distance on the ball, it follows  that for $\eps>\delta>0$ there exists $n_{R, \eps} \in \mathbb{N}$ such that for all $n \geq n_{R, \eps}$ the following holds: 
      \begin{align}
          K_{D}(\phi_{n}(x), \phi_{n}(y))\geq  K_{\bn}(\psi_{n}\circ\phi_{n}(x), \psi_{n}\circ\phi_{n}(y))\geq K_{\bn}(F(x), F(y))-\delta.
      \end{align}
 Since $F \in \text{Aut}(\bn)$, we conclude from the  above equation that $ K_{D}(\phi_{n}(x), \phi_{n}(y))\geq K_{\bn}(x, y)-\delta$ for all $x, y \in \overline{B}_{R}$ and for all $n \geq n_{R, \eps}$. The first part of \eqref{E:claim1} follows from the fact that holomorphic maps are non-expanding with respect to the Kobayashi distance. This proves the claim.     
 \end{proof}
\end{claim}
\noindent

Here $\overline{B}_{R}$ is compact subset of $\bn$ and according to our assumption, we have $\psi_{n} \circ \phi_{n}$ converges to $F \in \text{Aut}(\bn)$ uniformly on $\overline{B}_{R}$ as $n \to \infty$. Hence $\text{det}(D(\psi_{n} \circ \phi_{n})(z)) \to 1$ uniformly on $\overline{B_{R}}$ as $n \to \infty$.  Therefore, from holomorphic inverse function theorem, it follows that there exists $n'_{R, \eps} \in \mathbb{N}$ such that $ \phi_{n}$ are open maps on $B_{R}$ for all $n \geq n'_{R, \eps}$.  For the chosen compact set  $K$, we choose $R_{K}>0$ as Lemma~\ref{L:containofgeo}. We next choose  $n_{R_{K}, \eps} \in \mathbb{N}$ such that $\phi_{n}$ satisfies \eqref{E:claim1} for $R=R_{K}$ and $\delta=\frac{\eps}{2}$.    Therefore, we infer from Lemma~\ref{L:containofgeo} that $\gamma_{n}(K) \st B(\phi_{n}(B_{R_{K}}), \eps)$ for all $n \geq l_{R_{K}, \eps}:=\max\{n_{R_{K}, \eps}, n_{K}, n'_{R_{K}, \eps}\}$.

\medskip

Here we  claim  the following:
\begin{claim}\label{Claim2}
    For any $R, \eps>0$, there exits $M_{R, \eps} \in \mathbb{N}$ such that 
    \begin{align}
        K_{\bn}(\psi_{n}(z), \psi_{n}(w)) \geq K_{D}(z,w)-9\eps ~~\forall z, w
\in B_{D}(\phi_{n}(B_{R}), \eps)~~\forall n \geq M_{R, \eps}.    \end{align}
\end{claim}
\begin{proof}
    Let $z', w' \in B(\phi_{n}(B_{R}), \eps)$. Clearly, there exist $z_{n,\eps}, w_{n,\eps} \in B_{R}$ such that 
\begin{align}\label{E:dist0}
   K_{D}(z', \phi_{n}(z_{n,\eps}))&<\text{dist}(z', \phi_{n}(B_{R}))+\eps <2\eps, 
    \end{align}
\begin{align}\label{E:dist1}
    K_{D}(w', \phi_{n}(w_{n,\eps}))&<\text{dist}(w', \phi_{n}(B_{R}))+\eps <2\eps. 
    \end{align}
\noindent
We now deduce the following from triangle inequality:
\begin{align}\label{E:esti0}
 K_{\bn}(\psi_{n}(z'), \psi_{n}(w')) &\geq K_{\bn}(\psi_{n}\circ \phi_{n}(z_{n, \eps}),\psi_{n}\circ \phi_{n}(w_{n, \eps}))-K_{\bn}(\psi_{n}\circ\phi_{n}(z_{n,\eps}),\psi_{n}(z'))\notag\\&-K_{\bn}(\psi_{n}\circ\phi_{n}(w_{n,\eps}),\psi_{n}(w')), \notag\\
\intertext{since holomorphic maps are non-expansive with respect to the Kobayashi distance, we get}
 K_{\bn}(\psi_{n}(z'), \psi_{n}(w')) &\geq   K_{\bn}(\psi_{n}\circ \phi_{n}(z_{n, \eps}),\psi_{n}\circ \phi_{n}(w_{n, \eps}))-K_{D}(\phi_{n}(z_{n, \eps}),z')-K_{D}(\phi_{n}(w_{n, \eps}),w').
\end{align} 
\noindent
 It follows from  \eqref{E:dist0}, \eqref{E:dist1} and \eqref{E:esti0} that 
\begin{align}\label{E:estim}
 K_{D}(\psi_{n}(z'),\psi_{n}(w')) &  \geq   K_{\bn}(\psi_{n}\circ \phi_{n}(z_{n, \eps}),\psi_{n}\circ \phi_{n}(w_{n, \eps}))-4\eps. 
 \end{align}
\noindent
Since $\psi_{n} \circ \phi_{n}$ converges to an automorphism $F \in \text{Aut}(\bn)$ uniformly on $\overline{B}_{R}$ and automorphisms are Kobayashi isometry, hence, we get that there exists $M_{R, \eps} \in \mathbb{N}$ such that 
 \begin{align}\label{E:estimate1}
K_{D}(\psi_{n}(z'),\psi_{n}(w')) &\geq  K_{\bn}(z_{n, \eps}, w_{n, \eps})-5\eps ~~\forall n \geq M_{R, \eps}. 
 \end{align}
\noindent
From the triangle inequality,  we deduce that:
\begin{align*}
    K_{D}(z',w')&\leq K_{D}(z', \phi_{n}(z_{n,\eps}))+ K_{D}(\phi_{n}(z_{n,\eps}),\phi_{n}(w_{n,\eps}))+ K_{D}(w', \phi_{n}(w_{n,\eps}))
    \end{align*}
    using \eqref{E:dist0}, \eqref{E:dist1} we get the following from the above equation
    \begin{align}\label{E:estimate}
 K_{D}(z',w')    &\leq K_{D}(\phi_{n}(z_{n,\eps}),\phi_{n}(w_{n,\eps}))+4\eps \notag \\
    &\leq K_{\bn}(z_{n,\eps}, w_{n,\eps})+4\eps.
    \end{align}

Therefore, from \eqref{E:estimate},\eqref{E:estimate1} we obtain that 
\begin{align}
   K_{D}(\psi_{n}(z'),\psi_{n}(w')) &\geq K_{D}(z',w')-9\eps ~~~ \forall z',w' \in B(\phi_{n}(B_{R}),\eps)~~ \forall n \geq M_{R,\eps}.
\end{align}
This proves the Claim~\ref{Claim2}. 
\end{proof}
Now, let us go back to the proof of the lemma. We have that $\gamma_{n}(K) \st B(\phi_{n}(B_{R_{K}}), \eps)$ for all $n \geq l_{R_{K}, \eps}$. Hence, in view of Claim~\ref{Claim2}, we conclude that for all $n >N_{R, \eps}:=\max\{l_{R_{K}, \eps}, M_{R_{K}, \eps}\}$, for all $s, t \in K$ the following holds: 
\begin{align}\label{E:estimate4}
    K_{D}(\gamma_{n}(s), \gamma_{n}(t)) \geq K_{\bn}(\psi_{n}\circ\gamma_{n}(s),\psi_{n}\circ\gamma_{n}(t))&\geq K_{D}(\gamma_{n}(s),\gamma_{n}(t))-9\eps, \notag\\
    \intertext{since $\gamma_{n}$ is a  geodesic, we get}
 |s-t|\geq  K_{\bn}(\psi_{n}\circ\gamma_{n}(s),\psi_{n}\circ\gamma_{n}(t))&\geq |s-t|-9\eps.  
\end{align}
\noindent
Now letting $n \to \infty$ in \eqref{E:estimate4}, we get that 
\begin{align}\label{E:estimate5}|s-t|\geq  K_{\bn}(\eta(s),\eta(t))&\geq |s-t|-9\eps.
\end{align}
Since the compact set $K \st \mathbb{R}$   and $\eps>0$ is arbitrarily chosen, we conclude that 
$K_{\bn}(\eta(s),\eta(t))=|s-t|$ for all $s,t \in \mathbb{R}$. This proves the lemma.
\end{proof}
\begin{lemma}\label{L:stablegeo}
    Let $t_{n}>0, $ and   $\gamma_{n}:[-t_{n}, \infty) \to D$, $\psi_{n} \in \mco(D, \bn)$  and  $\phi_{n} \in \mco(\bn, D)$ as Lemma~\ref{L:convergeo}. If $\tilde{\gamma}_{n}\colon [-t_{n}, \infty) \to D$ is a sequence of geodesic rays  such that $\sup_{t \in [-t_{n}, \infty)}K_{D}(\gamma_{n}(t), \tilde{\gamma}_{n}(t))< \infty$ for all $ n \in \mathbb{N}$, then $\psi_{n} \circ \tilde{\gamma}_{n} \to \tilde{\eta}$ (upto a subsequence) uniformly over every compact subset of $\mathbb{R}$ to a geodesic ray $\tilde{\eta}: \mathbb{R} \to \bn$.
\end{lemma}
\begin{proof}
    Let $K \st \re$ be any compact subset. Clearly, there exists $n_{K} \in \mathbb{N}$ such that $K \st \text{domain}(\tilde{\gamma}_{n})$ for all $n \geq n_{k}$.  It follows from  Arzel\`a-Ascoli theorem that the sequence $\psi_{n}\circ \tilde{\gamma}_{n} $ has a convergent subsequence that converges uniformly to some function on $K$. Renaming the subsequence, we assume that that $\psi_{n} \circ \tilde{\gamma}_{n}$ uniformly converges on $K$. Now define $\tilde{\eta}:\re \to D$ by  $\tilde{\eta}(t):=\lim_{n \to \infty}\psi_{n}\circ \tilde{\gamma}_{n}(t)$. We need to show that $\tilde{\eta}$ is a geodesic line.  To this aim, we prove the following claim:
    \begin{claim}\label{Claim3}
    For all  compact  subset $K \st \re$ and $\eps>0$ there exists $R_{K}>0$ and $m_{K, \eps} \in \mathbb{N}$ such that $\tilde{\gamma}_{n}(K) \st B(\phi_{n}(B_{R_{K}}), \eps)$ for all $n \geq m_{R_{K}, \eps}$.
    \end{claim}

    \begin{proof}

Let $L:=\sup_{t \in [-t_{n}, \infty)}K_{D}(\gamma_{n}(t), \tilde{\gamma}_{n}(t))$ and $R_{K}>r+L+\eps$ where $r=\sup_{t \in K}|t|>0$. Proceeding in a similar way as Lemma~\ref{L:convergeo}, we choose $m_{R, \eps} \in \mathbb{N}$ such that $\phi_{n}$ satisfies the conditions of Lemma~\ref{L:containmentball} for all $n >m_{R_{K}, \eps}$. We have that  
\begin{align}
    K_{D}(\phi_{n}(0), \tilde{\gamma}_{n}(t))&\leq K_{D}(\tilde{\gamma}_{n}(t),\gamma_{n}(t))+K_{D}(\gamma_{n}(t),\gamma_{n}(0))\notag\\
    &\leq L+r<R_{K}-\eps .
\end{align}
Therefore, from Lemma~\ref{L:containofgeo} 
 $\tilde{\gamma}_{n}(K) \subset B(\phi_{n}(B_{R_{k}}), \eps)$ for all $n \geq m_{K, \eps}$. This proves the claim.
\end{proof}
Now invoking Claim~\ref{Claim2} and  proceeding same way as Lemma~\ref{L:containofgeo} we conclude that $\psi_{n}\circ \tilde{\gamma}_{n} \to \tilde{\eta}$ uniformly on every compact subset of $\re$ and $\tilde{\eta}:\re \to \bn$. This proves the lemma.
    \end{proof}

\begin{lemma}\label{L:differgeo}
    Let $ \gamma_{n}, \eta,  \phi_{n}, \psi_{n}$ be as Lemma~\ref{L:convergeo} and $\tilde{\gamma}_{n}$ $\tilde{\eta}$ be as Lemma~\ref{L:stablegeo}. If there exists $c>0$ such that $0<c \leq \inf_{s \geq 0, n \in \mathbb{N}}K_{D}(\gamma_{n}(0), \tilde{\gamma}_{n}(s))$, then $\eta$ and $\tilde{\eta}$ are two distinct geodesic lines.
\end{lemma}
\begin{proof}
    Let $K =[0,1]$ and $\eps \in (0, \frac{c}{9})$. It follows from the  proof of Lemma~\ref{L:convergeo} and  Claim~\ref{Claim3}  that there exists $R_{K}>0$ and $\widetilde{N}_{R_{K}, \eps}=\max\{m_{R_{K},\eps},l_{R_{K}, \eps}\}$ such that both $\gamma_{n}[0,1]$ and $\tilde{\gamma}_{n}[0,1]$ are inside $B(\phi_{n}(B_{R_{K}}), \eps)$ for all $n \geq \widetilde{N}_{R_{K}, \eps} $. Therefore, in view of Claim~\ref{Claim2} we conclude that the following holds for all $t \in [0,1]$ and for all $n \geq \widetilde{N}_{R_{K}, \eps}$
    \begin{align}\label{E:estim5}
        K_{\bn}(\psi_{n}\circ \gamma_{n}(0),\psi_{n}\circ \tilde{\gamma}_{n}(t))&\geq K_{D}(\gamma_{n}(0), \tilde{\gamma}_{n}(t))-9\eps \notag\\
        &\geq \inf_{s \geq 0}K_{D}(\gamma_{n}(0), \tilde{\gamma}_{n}(s))-9\eps \notag\\
        &>c-9\eps>0.
    \end{align}
 Letting $n \to \infty$ in \eqref{E:estim5} we get that $K_{\bn}(\eta(0), \tilde{\eta}(t)) \geq c-9\eps>0$ for all $t \in [0,1]$. Therefore, $\eta$ and $\tilde{\eta}$ are two distinct geodesic lines.
 \end{proof}

 \begin{lemma}\label{L:boundline}

Let $ \gamma_{n}, \eta, \tilde{\eta}, \phi_{n}, \psi_{n}$ be as Lemma~\ref{L:differgeo}. Then there exists $C>0$ such that $\sup_{t \in \re}K_{\bn}(\eta(t), \tilde{\eta}(t)) <C$.
 \end{lemma}
\begin{proof}
   For any $M >0$, we have  that $\psi_{n} \circ \gamma_{n} \to \eta$ and $\psi_{n} \circ \tilde{\gamma} \to \tilde{\eta}$ uniformly on $[-M, M]$. Since holomorphic maps are non-expansive with respect to Kobayashi distance, hence we have that 
   $$K_{\bn}(\psi_{n} \circ \gamma_{n}(t),\psi_{n} \circ \tilde{\gamma}_{n} (t)) \leq K_{D}(\gamma_{n}(t), \tilde{\gamma}_{n}(t))<C,~~ \forall t \in [-M, M].$$
   Letting $n \to \infty$ in the above equation, we get that $K_{\bn}(\eta(t), \tilde{\eta}(t))\leq C$ for all $t \in [-M, M]$.
Since the relation holds for all $M >0$, we conclude the lemma. 
\end{proof}
We now give the proof of Theorem~\ref{T:approachsgeo}.
\begin{proof}[Proof of Theorem~\ref{T:approachsgeo}] 
If  the theorem is not true, then it follows from Result~\ref{R:appgeocondition}  that there exists two geodesic rays $\gamma_{j}\colon [0, \infty) \to D$ for $j=\{1,2\}$, a sequence of positive real number $t_{n}$ with $t_{n} \to \infty$ as $n \to  \infty$ and constants   $C>c>0$ such that 
    \begin{align}\label{E:gromovbound}
        \sup_{t \geq 0} K_{D}(\gamma_{1}(t), \gamma_{2}(t))<C,
    \end{align}
    and 
    \begin{align}\label{E:appogeocon1}
        0<c\leq \inf_{s \geq 0}K_{D}(\gamma_{1}(t_{n}), \gamma_{2}(s))\leq K_{D}(\gamma_{1}(t_{n}), \gamma_{2}(t_{n}))\leq C.
    \end{align}

\noindent
 From our assumption, $z_{n}:=\gamma_{1}(t_{n}) \to \partial_{H} D$ as $n \to \infty$ and $q_{D}(z_{n}) \to 1$ as $n \to \infty$. It follows from  Remark~\ref{remark} that there exists 
$\psi_{n} \in \mco(D, \bn)$ with $\psi_{n}(z_{n})=0$ and $\phi_{n} \in \mco (\bn, D)$ with $\phi_{n}(0)=z_{n}$ such that $q_{D}(z_{n})=|\text{det}(\psi_{n}'(z_{n}))|^{2}|\text{det}(\phi_{n}'(0))|^{2}$. Therefore, given any $\delta>0$  there exists $n_{\delta} \in \mathbb{N}$ such that  
\begin{align}
|\text{det}\psi_{n}'(\phi_{n}(0))||\text{det}\phi_{n}'(0)| \in (1-\delta, 1]~~~~\forall~~ n \geq n_{\delta}.
\end{align}
Consequently,
\begin{align}\label{E:determinat}
    |\text{det}(\psi_{n}\circ\phi_{n})'(0) | \in (1-\delta, 1]~~~~ \forall n \geq n_{\delta}.
\end{align}
Here $\psi_{n} \circ \psi_{n}(0)=0$ for all $n \in \mathbb{N}$. Hence, by Montel's theorem, there exists a convergent subsequence of $\psi_{n} \circ \phi_{n}$. After passing to a subsequence we can assume that the sequence $\psi_{n} \circ \phi_{n}$ converges to a $F \in \mco(\bn, \bn)$ uniformly over every compact subset of $\bn$. Since $\delta>0$ is arbitrary in \eqref{E:determinat}, hence letting $n \to \infty$ we conclude that $|\text{det}F(0)|=1$.  We now invoke Result~\ref{R:Rudin} to conclude that $F \in \text{Aut}(\bn)$. Let  us  consider two geodesic sequences denoted  by 
$\sigma_{j, n} \colon [-t_{n}, \infty) \to D$ and is defined by $\sigma_{j, n}(t):=\gamma_{j}(t+t_{n})$ for $j =\{1,2\}$. Clearly,  $\psi_{n} \circ \sigma_{1,n}(0)=0$. Taking $\gamma_{n}=\sigma_{1,n}$ in Lemma~\ref{L:convergeo}, we conclude that $\psi_{n}\circ\sigma_{1,n}:[-t_{n}, \infty) \to D$ uniformly over every compact subsets of $\re$ to a geodesic line $\eta \colon \re \to \bn$. In view of \eqref{E:gromovbound} we get that 
\begin{align*}
    K_{D}(\sigma_{1,n}(t),(\sigma_{2,n}(t))=K_{D}(\gamma_{1}(t+t_{n}),\gamma_{2}(t+t_{n}))\leq \sup_{t\geq 0}K_{D}(\gamma_{1}(t),\gamma_{2}(t))<C, ~~\forall n \in \mathbb{N}.
\end{align*}
Therefore, using Lemma~\ref{L:stablegeo} we conclude  that $\psi_{n}\circ\sigma_{2,n}\colon[-t_{n}, \infty) \to \bn$ uniformly over every compact subsets of $\re$ to a geodesic line $\tilde{\eta}\colon \re \to \bn$. From \eqref{E:appogeocon1}, we get that \begin{align}
    \inf_{s \geq 0}K_{D}(\sigma_{1,n}(0),\sigma_{2,n}(s))\geq c>0~~~\forall n \in \mathbb{N}.
\end{align}
Therefore, from Lemma~\ref{L:differgeo} we conclude that two geodesic lines $\eta$ and $\tilde{\eta}$ are distinct.  Finally from Lemma~\ref{L:boundline}, we get that there exists $\zeta \neq \xi \in \partial \bn$ such that $\lim_{t\to \infty}\eta(t)=\lim_{t\to \infty}\tilde{\eta}(t)=\zeta$ and   $\lim_{t\to -\infty}\eta(t)=\lim_{t\to -\infty}\tilde{\eta}(t)=\xi$. This leads to a contradiction. Therefore, we conclude the theorem.
\end{proof}

\section{Consequences of approaching geodesics}

Here we present the proof of the Theorem ~ \ref{T:horosphere}.

\begin{proof}[Proof of Theorem~\ref{T:horosphere}]
 In view of Result~\ref{R:Zhimmersimpleboundary}, it is enough to show that $\overline{\text{Aff}(\cn)\cdot D}\cap \mathbb{X}_{d}$ has simple boundary.  Let $D' \in \overline{\text{Aff}(\cn)\cdot D}\cap \mathbb{X}_{d}\setminus \text{Aff}(\cn)\cdot D$. Then there exists $z_{n} \in D$ such that $ z_{n} \to  \partial D$ as $n \to \infty$ and $A_{n} \in \text{Aff}(\cn)$ such that $(A_{n}D, A_{n}z_{n}) \to (D',z_{0})$ in $\mathbb{X}_{d,0}$ in local Hausdorff convergence. We  now infer from Result~\ref{R:upper semicon} that  
 \begin{align*}
1\geq q_{D'}(z_{0})\geq  \limsup_{n \to \infty} q_{A_{n}D}(A_{n}z_{n})=\limsup_{n \to \infty}q_{D}(z_{n})=1. 
 \end{align*}
\noindent
Therefore, it follows that $D'$ is biholomorphic to the unit ball in $\cn$. Hence, it has a simple boundary. The rest of the argument follows from a similar argument as in the proof of \cite[Lemma 20.8]{Zimmersubelip2022}.  Now we infer from Result~\ref{R:bracci1} and Result~\ref{R:bracci2} that the domain  $D$ has the visibility property with respect to the Kobayashi distance. Here $q_{D}(z) \to 1$ as $z \to \partial D$. By a similar argument, as Theorem~\ref{T:approachsgeo} we conclude  that $(D,K_{D})$ has the approaching geodesic property, hence, invoking, Result~\ref{R:extn of busemanmap} we conclude that there is a homeomorphism between $\overline{D}^{G}$ and $\overline{D}^{H}$ extending  the identity map. Therefore, we are done.
\end{proof}

\noindent {\bf Acknowledgements.} 
The work of the second-named author is supported by the Institute Postdoctoral Fellowship of the Indian Institute of Technology, Kanpur.


\begin{thebibliography}{10}


\bib{Abate1988Z}{article}{
      author={Abate, Marco},
       title={Horospheres and iterates of holomorphic maps},
        date={1988},
        ISSN={0025-5874,1432-1823},
     journal={Math. Z.},
      volume={198},
      number={2},
       pages={225\ndash 238},
         url={https://doi.org/10.1007/BF01163293},
      review={\MR{939538}},
}

\bib{Abateiterationtheorynote}{book}{
      author={Abate, Marco},
       title={Iteration theory of holomorphic maps on taut manifolds},
      series={Research and Lecture Notes in Mathematics. Complex Analysis and Geometry},
   publisher={Mediterranean Press, Rende},
        date={1989},
      review={\MR{1098711}},
}

\bib{Abate1990}{article}{
      author={Abate, Marco},
       title={The {L}indel\"of principle and the angular derivative in strongly convex domains},
        date={1990},
        ISSN={0021-7670,1565-8538},
     journal={J. Analyse Math.},
      volume={54},
       pages={189\ndash 228},
         url={https://doi.org/10.1007/BF02796148},
      review={\MR{1041181}},
}

\bib{AroFia2025}{article}{
      author={Arosio, Leandro},
      author={Fiacchi, Matteo},
       title={On the approaching geodesics property},
        date={2025},
        ISSN={1972-6724,2198-2759},
     journal={Boll. Unione Mat. Ital.},
      volume={18},
      number={1},
       pages={3\ndash 16},
         url={https://doi.org/10.1007/s40574-024-00442-7},
      review={\MR{4871877}},
}

\bib{AFG2024}{article}{
      author={Arosio, Leandro},
      author={Fiacchi, Matteo},
      author={Gontard, S\'ebastien},
      author={Guerini, Lorenzo},
       title={The horofunction boundary of a {G}romov hyperbolic space},
        date={2024},
        ISSN={0025-5831,1432-1807},
     journal={Math. Ann.},
      volume={388},
      number={2},
       pages={1163\ndash 1204},
         url={https://doi.org/10.1007/s00208-022-02551-0},
      review={\MR{4700367}},
}

\bib{DevBorah2020}{article}{
      author={Borah, Diganta},
      author={Kar, Debaprasanna},
       title={Boundary behavior of the {C}arath\'eodory and {K}obayashi-{E}isenman volume elements},
        date={2020},
        ISSN={0019-2082,1945-6581},
     journal={Illinois J. Math.},
      volume={64},
      number={2},
       pages={151\ndash 168},
         url={https://doi.org/10.1215/00192082-8303461},
      review={\MR{4092953}},
}

\bib{BGZ2021}{article}{
      author={Bracci, Filippo},
      author={Gaussier, Herv\'{e}},
      author={Zimmer, Andrew},
       title={Homeomorphic extension of quasi-isometries for convex domains in {$\Bbb C^d$} and iteration theory},
        date={2021},
     journal={Math. Ann.},
      volume={379},
      number={1-2},
       pages={691\ndash 718},
}

\bib{BNT2022}{article}{
      author={Bracci, Filippo},
      author={Nikolov, Nikolai},
      author={Thomas, Pascal~J.},
       title={Visibility of {K}obayashi geodesics in convex domains and related properties},
        date={2022},
     journal={Math. Z.},
      volume={301},
      number={2},
       pages={2011\ndash 2035},
}

\bib{DGZ2016}{article}{
      author={Deng, Fusheng},
      author={Guan, Qi'an},
      author={Zhang, Liyou},
       title={Properties of squeezing functions and global transformations of bounded domains},
        date={2016},
        ISSN={0002-9947,1088-6850},
     journal={Trans. Amer. Math. Soc.},
      volume={368},
      number={4},
       pages={2679\ndash 2696},
         url={https://doi.org/10.1090/tran/6403},
      review={\MR{3449253}},
}

\bib{GWU1985}{article}{
      author={Graham, Ian},
      author={Wu, H.},
       title={Characterizations of the unit ball {$B^n$} in complex {E}uclidean space},
        date={1985},
        ISSN={0025-5874,1432-1823},
     journal={Math. Z.},
      volume={189},
      number={4},
       pages={449\ndash 456},
         url={https://doi.org/10.1007/BF01168151},
      review={\MR{786275}},
}

\bib{Gromov:1981}{article}{
      author={Gromov, M},
       title={Hyperbolic manifolds, groups and actions, in: Riemann surfaces and related topics},
        date={1981},
     journal={Proceedings of the 1978 Stony Brook Conference Ann.\,of Math.\,Stud., 97, Princeton Univ.\,Press},
}

\bib{Krantzbook}{book}{
      author={Krantz, Steven~G.},
       title={Function theory of several complex variables},
   publisher={AMS Chelsea Publishing, Providence, RI},
        date={2001},
        ISBN={0-8218-2724-3},
         url={https://doi.org/10.1090/chel/340},
        note={Reprint of the 1992 edition},
      review={\MR{1846625}},
}

\bib{Rudin2008}{book}{
      author={Rudin, W.},
       title={Function theory in the unit ball of {$\Bbb C^n$}},
      series={Classics in Mathematics},
   publisher={Springer-Verlag, Berlin},
        date={2008},
        note={Reprint of the 1980 edition},
}

\bib{Zimmersubelip2022}{article}{
      author={Zimmer, Andrew},
       title={Subelliptic estimates from {G}romov hyperbolicity},
        date={2022},
        ISSN={0001-8708,1090-2082},
     journal={Adv. Math.},
      volume={402},
       pages={Paper No. 108334, 94},
         url={https://doi.org/10.1016/j.aim.2022.108334},
      review={\MR{4397690}},
}


\end{thebibliography}
	
 \end{document}